\documentclass[11pt]{amsart}
\usepackage{latexsym, mathtools}
\usepackage{scalefnt}
\usepackage{rotating}
\usepackage{amsfonts}
\usepackage{amssymb}
\usepackage{amsmath, stmaryrd}
\usepackage{stackrel}
\usepackage{graphics, color}
\definecolor{darkblue}{rgb}{0,0,0.4}
\usepackage[colorlinks=true, citecolor=darkblue, filecolor=darkblue,
linkcolor=darkblue, urlcolor=darkblue]{hyperref}
\input xy
\xyoption{all}
\usepackage{tikz-cd}
\usepackage{url}

\makeatletter
\def\amsbb{\use@mathgroup \M@U \symAMSb}
\makeatother
\usepackage{xpatch}

\xpatchcmd{\proof}
  {\itshape}
  {\bfseries}
  {}
  {}
\usepackage{bbold}

\makeatletter
\newcommand*{\@old@slash}{}\let\@old@slash\slash
\def\slash{\relax\ifmmode\delimiter"502F30E\mathopen{}\else\@old@slash\fi}
\makeatother

\newtheorem{theorem}{Theorem}[section]
\newtheorem{lemma}[theorem]{Lemma}
\newtheorem{corollary}[theorem]{Corollary}
\newtheorem{proposition}[theorem]{Proposition}
\newtheorem{claim}[theorem]{Claim}

\newtheorem{example}[theorem]{Example}
\newtheorem{definition}[theorem]{Definition}
\theoremstyle{remark}
\newtheorem{remark}[theorem]{Remark}

\newsavebox{\cgc}

\newsavebox{\Dxy}
\title[The Homotopy Type of Spaces of Flat Connections]{The Homotopy Type of Spaces of Flat Connections for Classical Lie Groups}
\author{Andrew Davis}
\begin{document}
\begin{abstract}
    Let $M$ be a smooth manifold. We use Chern-Weil theory to study the characteristic classes of principal $G$-bundles built from continuous families of $\pi_{1}(M)$-representations, where $G$ is a compact Lie group. We then relate these families to the functorial map $$\text{Hom}(\pi_{1}(M), G)\rightarrow\text{Map}_{*}(M,BG)$$ and use this relationship to study the weak homotopy type of the space of flat connections for $U(n)$, $O(n)$, $SO(n)$, and Spin($n$) bundles. 
\end{abstract}
\maketitle

\renewcommand{\v}[1]{{\mathbf #1}}

\newcommand{\rk}{\mathrm{rk}}
\newcommand{\bbA}{\asmbb{A}}
\newcommand{\bbC}{\amsbb{C}}
\newcommand{\C}{\amsbb{C}}
\newcommand{\N}{\amsbb{N}}
\newcommand{\bbN}{\amsbb{N}}
\newcommand{\R}{\amsbb{R}}
\newcommand{\bbR}{\amsbb{R}}
\newcommand{\bbRP}{\amsbb{RP}}
\newcommand{\Z}{\amsbb{Z}}
\newcommand{\bbZ}{\amsbb{Z}}
\newcommand{\mcZ}{\mathcal{Z}}
\newcommand{\bbF}{\amsbb{F}}
\newcommand{\bbK}{\amsbb{K}}
\newcommand{\bbQ}{\amsbb{Q}}
\newcommand{\ignore}[1]{}
\newcommand{\A}{\mathcal{A}}
\newcommand{\B}{\mathcal{B}}
\newcommand{\mcC}{\mathcal{C}}
\newcommand{\mcD}{\mathcal{D}}
\newcommand{\D}{\mathcal{D}}
\newcommand{\E}{\mathcal{E}}
\newcommand{\mF}{\mathcal{F}}
\newcommand{\G}{\mathcal{G}}
\newcommand{\mcG}{\mathcal{G}}
\newcommand{\mH}{\mathcal{H}}
\newcommand{\mcH}{\mathcal{H}} 
\newcommand{\I}{\mathcal{I}}
\newcommand{\mL}{\mathcal{L}}
\newcommand{\mcN}{\mathcal{N}}
\newcommand{\M}{\mathcal{M}}
\newcommand{\mO}{\mathcal{O}}
\newcommand{\mcP}{\mathcal{P}}
\newcommand{\mcR}{\mathcal{R}}
\newcommand{\qcd}{\bbQ cd}
\newcommand{\FP}{\textrm{FP}}

\newcommand{\T}{\mathcal{T}}
\newcommand{\U}{\mathcal{U}}
\newcommand{\V}{\mathcal{V}}

\newcommand{\bS}{\mathbf{S}}
\newcommand{\bT}{\mathbf{T}}
\newcommand{\bd}{\mathbf{d}}

\newcommand{\Int}{\textrm{Int}}
\newcommand{\srmx}[1]{\xmaps{#1}}
\newcommand{\sth}{\textrm{th}}
\newcommand{\srlm}[1]{\stackrel{#1}{\lmaps}}
\newcommand{\Spectra}{\mathbf{Spectra}}
\newcommand{\bG}{\mathcal{G}_0} 
\newcommand{\abs}[1]{\left| #1\right|}
\newcommand{\ord}[1]{\Delta \left( #1 \right)}
\newcommand{\leqs}{\leqslant}
\newcommand{\geqs}{\geqslant}
\newcommand{\heq}{\simeq}
\newcommand{\iso}{\simeq}
\newcommand{\maps}{\longrightarrow}
\newcommand{\lmaps}{\longleftarrow}
\newcommand{\injects}{\hookrightarrow}
\newcommand{\homeo}{\cong}
\newcommand{\surjects}{\twoheadrightarrow}
\newcommand{\isom}{\cong}
\newcommand{\cross}{\times}
\newcommand{\normal}{\vartriangleleft}
\newcommand{\wt}[1]{\widetilde{#1}} 
\newcommand{\fc}{\mathcal{A}_{\mathrm{flat}}} 
\newcommand{\flc}{\mathcal{A}_{\mathrm{fl}}}

\newcommand{\Ch}{\mathrm{Ch}}
\newcommand{\Rdef}{R^{\mathrm{def}}}
\newcommand{\Sym}{\textrm{Sym}}
\newcommand{\vect}[1]{\stackrel{\rightharpoonup}{\mathbf #1}}
\newcommand{\SR}{\mathcal{SR}}
\newcommand{\SRe}{\mathcal{SR}^{\mathrm{even}}}
\newcommand{\Rep}{\mathrm{Rep}}
\newcommand{\SRep}{\mathrm{SRep}}
\newcommand{\Hom}{\mathrm{Hom}}
\newcommand{\HHom}{\mathcal{H}\mathrm{om}}
\newcommand{\Lie}{\mathrm{Lie}}
\newcommand{\GL}{\mathrm{GL}}
\newcommand{\K}{K^{\mathrm{def}}}
\newcommand{\mK}{\mathcal{K}_{\mathrm{def}}}
\newcommand{\SK}{SK_{\mathrm{def}}}
\newcommand{\dom}{\mathrm{dom}}
\newcommand{\codom}{\mathrm{codom}}
\newcommand{\Ob}{\mathrm{Ob}}
\newcommand{\Mor}{\mathrm{Mor}}
\newcommand{\+}[1]{\underline{#1}_+}
\newcommand{\Fin}{\Gamma^{\mathrm{op}}}
\newcommand{\f}[1]{\underline{#1}}
\newcommand{\hofib}{\mathrm{hofib}}
\newcommand{\Stab}{\mathrm{Stab}}
\newcommand{\Css}{\mathcal{C}_{ss}}
\newcommand{\Map}{\mathrm{Map}}
\newcommand{\bMap}{\mathrm{Map_*}}
\newcommand{\bdMap}{\mathrm{Map_*^\delta}}
\newcommand{\flatc}{\mathcal{A}_{\mathrm{flat}}}
\newcommand{\F}[1]{\mathrm{Flag}(\vect{#1})}
\newcommand{\p}{\vect{p}}
\newcommand{\avg}{\mathrm{avg}}
\newcommand{\smsh}[1]{\ensuremath{\mathop{\wedge}_{#1}}}
\newcommand{\Vect}{\mathrm{Vect}}
\newcommand{\bv}{\bigvee}
\newcommand{\Gr}{\mathrm{Gr}}
\newcommand{\Mf}{\mathcal{M}_{\textrm{flat}}}
\newcommand{\ku}{\mathbf{ku}}
\newcommand{\Susp}{\Sigma}
\newcommand{\Id}{\mathrm{Id}}
\newcommand{\id}{\textrm{Id}}
\newcommand{\xmaps}{\xrightarrow}
\newcommand{\srm}[1]{\stackrel{#1}{\maps}}
\newcommand{\srt}[1]{\stackrel{#1}{\to}}
\newcommand{\sm}{\wedge}
\newcommand{\conv}{\Rightarrow}
\newcommand{\Tor}{\textrm{Tor}}
\newcommand{\goesto}{\mapsto}
\newcommand{\nd}{\noindent}
\newcommand{\lie}[1]{\mathcal{#1}}
\newcommand{\rK}{K^\textrm{rep}}
\newcommand{\wtrK}{\wt{K}^\textrm{rep}}
\newcommand{\rKU}{KU^\textrm{rep}}
\newcommand{\wtrKU}{\wt{KU}^\mathrm{rep}}
\newcommand{\cc}{\Box}
\newcommand{\ol}[1]{\overline{#1}}
\def\co{\colon\thinspace}

\newcommand{\Span}{\mathrm{Span}}
\newcommand{\Img}{\mathrm{Im}}
 \newcommand{\s}[1]{\vspace{#1 in}}
\newcommand{\e}{\emph}
\newcommand{\vv}{\vec{v}}
\newcommand{\ww}{\vec{w}}
\newcommand{\ee}{\vec{e}}
\newcommand{\xx}{\vec{x}}
\newcommand{\bb}{\vec{b}}
 \newcommand{\ra}{\rangle}
\newcommand{\la}{\langle}
 
\newcommand{\defn}{\mathrel{\mathop :}=}
\newcommand{\psubset}{\subsetneqq}
\newcommand{\propersubset}{\subsetneqq}
\newcommand{\tR}{\wt{\mcR}}
\newcommand{\coker}{\textrm{coker}}

\section{Introduction}

Let $\Gamma$ be a discrete group such that $B\Gamma$ has the homotopy type of a finite CW complex. Baird and Ramras \cite{bairdramras} used differential and homotopy theoretic methods to study the functorial map $$\text{Hom}(\Gamma,GL_{n}(\bbC))\xrightarrow{B}\text{Map}_{*}(B\Gamma,BGL_{n}(\bbC))$$ from the space of complex representations to the space of based continuous maps between classifying spaces. This map is closely related to the holonomy of flat connections on complex vector bundles. Notably, Baird and Ramras were able to show the following:

\begin{theorem}[\cite{bairdramras} Corollary 1.3]\label{rbthm}
    Let M be a closed, smooth, aspherical manifold of dimension d, and let $E\rightarrow M$ be a flat principal $U(n)$-bundle over M. Let $\mathcal{A}_{\mathrm{flat}}(E)$ denote the space of flat, unitary connections on E. Then for each $A_0\in \mathcal{A}_{\mathrm{flat}}(E)$ and each $m$ such that $0< m\leq 2n-d-1$, we have $$\mathrm{Rank}(\pi_{m}(\mathcal{A}_{\mathrm{flat}}(E),A_0))\geq \sum_{i=1}^{\infty}\beta_{m+2i+1}(M).$$ In addition, if $\sum_{i=1}^{\infty}\beta_{2i+1}(M)>0$ and $2n-d-1\geq 0$, then $\mathcal{A}_{\mathrm{flat}}(E)$ has infinitely many path components. 
\end{theorem}
This is in contrast to the situation for complex vector bundles over surfaces, where Yang-Mills theory shows that the space of flat connections is highly connected with respect to the rank of the bundle (see \cite{ramras-yangmills} Proposition 4.9).\\

In this paper we generalize the result of Baird and Ramras by dropping the asphericity condition on the base manifold and extending beyond complex vector bundles to the case of real vector bundles and principal $SU(n)$, $SO(n)$, and Spin($n$) bundles, leading to the following lower bounds:

\begin{theorem}\label{introthm1}
    Let M be closed, smooth, connected manifold of dimension d, and let $E\rightarrow M$ be a flat principal $G$-bundle over M, where $G$ is $U(n), O(n), SU(n), SO(n),$ or Spin($n$). 
    
    \begin{itemize} 
    
    \item If $G=O(n)$, $ SO(n),$ or Spin$(n)$, then for each $A_0\in \mathcal{A}_{\mathrm{flat}}(E)$ and each $m$ such that $0< m\leq n-d-2$, $$\mathrm{Rank}(\pi_{m}(\mathcal{A}_{\mathrm{flat}}(E),A_0))\geq \sum_{i=1}^{\infty}\beta_{3m+4i+1}(M).$$ In addition, if $\sum_{i=1}^{\infty}\beta_{4i+1}(M)>0$ and $n-d-2\geq 0$, then $\mathcal{A}_{\mathrm{flat}}(E)$ has infinitely many path components. 

    \item If $G=SU(n)$, then for each $A_0\in \mathcal{A}_{\mathrm{flat}}(E)$ and each $m$ such that $0< m\leq 2n-d-1$, $$\mathrm{Rank}(\pi_{m}(\mathcal{A}_{\mathrm{flat}}(E),A_0))\geq \sum_{i=2}^{\infty}\beta_{m+2i+1}(M).$$ In addition, if $\sum_{i=2}^{\infty}\beta_{2i+1}(M)>0$ and $2n-d-1\geq 0$, then $\mathcal{A}_{\mathrm{flat}}(E)$ has infinitely many path components.

    \item If $G=U(n)$, we achieve the same bound as $G=SU(n)$ except the sum of Betti numbers begins at $i=1$.
    \end{itemize}
\end{theorem}


\noindent To prove this we use the following theorem:

\begin{theorem}[Section 3]\label{lbl9}
   Let $M$ be a closed, connected, smooth manifold of dimension d with fundamental group $\Gamma=\pi_{1}(M,m_0)$, let $f:M\rightarrow B\Gamma$ be a map classifying the universal cover $\widetilde{M}\rightarrow M$ that takes the basepoint $m_0$ to the canonical basepoint of $B\Gamma$, and let $E$ be a flat principal $G$-bundle over $M$ with $G$ a compact Lie group. Then there is a weak equivalence $$\mathcal{A}_{\mathrm{flat}}(E)\simeq \hofib_{B(\mathcal{H}(A))\circ f}(f^{*}\circ B)$$ where $\mathcal{H}(A)\in \mathrm{Hom}(\pi_{1}(M),G)$ is any representation in the image of the holonomy map $\mathcal{A}_{\mathrm{flat}}(E)\xrightarrow{\mathcal{H}} \mathrm{Hom}(\pi_{1}(M),G)$ and $f^{*}\circ B$ is the composite \[\mathrm{Hom}(\Gamma,G)\xrightarrow{B}\mathrm{Map}_{*}(B\Gamma,BG)\xrightarrow{f^{*}}\mathrm{Map}_{*}(M,BG).\] 
\end{theorem}

This allows us to study the space of flat connections on a $G$-bundle by understanding the weak homotopy type of $\hofib_{B(\mathcal{H}(A))\circ f}(f^{*}\circ B)$. In Section 2 we establish characteristic class restrictions on $G$-bundles classified by maps in the image of $f^{*}\circ B$. Then in Sections 3-5 we show that for the Lie groups $U(n)$, $SU(n)$, $O(n)$, $SO(n)$, and Spin$(n)$ these restrictions directly translate to information about the homotopy fiber.  

\subsection*{Acknowledgements} The author would like to thank Daniel Ramras for his generous support in my understanding and exploration of this topic.  
\section{Characteristic Classes of Generalized Flat Families}

Baird and Ramras observed that elements in $\pi_{*}(\text{Map}_{*}(B\Gamma,BGL_{n}(\bbC)))$ that are in the image of $$B_{*}:\pi_{*}(\text{Hom}(\Gamma,GL_{n}(\bbC)))\rightarrow \pi_{*}(\text{Map}_{*}(B\Gamma,BGL_{n}(\bbC)))$$ classify bundles built from spherical families of representations, and used Chern-Weil theory to show that the characteristic classes of these bundles vanish rationally in high degrees. In this section, we extend these results to general Lie groups.\\

First we recall some important definitions and results from Chern-Weil theory (see \cite{dupont} for full details). Let $G$ be a Lie group with Lie algebra $\mathfrak{g}$ and let $M$ be a smooth manifold. Note that we will be using $d\cdot$ and $(\cdot)_*$ interchangeably as notation for the differential of a map. 

\begin{definition}\label{def1}
A connection on a principal $G$-bundle $\pi:E\rightarrow M$ is a Lie algebra valued one-form $\theta \in \Omega^{1}(E;\mathfrak{g})$ such that:

\begin{enumerate}
    \item $\theta_x\circ v_x=$id, where $v_x:\mathfrak{g}\rightarrow T_{x}(E)$ is the differential of the map $h\mapsto x\cdot h$ at $e\in G$. We will call vectors in the image of $v_x$ vertical vectors. 
    \item $R_{g}^{*}(\theta)=$Ad$(g^{-1})\circ \theta$ for all $g\in G$, where $R_g: E\rightarrow E$ is the $G$-action on $E$ and $Ad(g):\mathfrak{g}\rightarrow \mathfrak{g}$ is the differential at the identity of the map $x\mapsto gxg^{-1}$. 
\end{enumerate}

We call vectors in $\mathrm{ker}(\theta)$ horizontal.
\end{definition}

\begin{definition}
    The curvature of a connection $\theta$ is the Lie algebra valued two-form $\Omega=d\theta+\frac{1}{2} [\theta,\theta]\in \Omega^{2}(E;\mathfrak{g})$, where $[\theta,\theta]$ is the image of $\theta \wedge \theta$ under the map \\ $\Omega^{2}(E;\mathfrak{g}\otimes\mathfrak{g})\rightarrow \Omega^{2}(E;\mathfrak{g})$ induced by the Lie bracket $[\hspace{1mm},\hspace{1mm}]:\mathfrak{g}\otimes \mathfrak{g}\rightarrow \mathfrak{g}$.\\ $\theta$ is called flat if $\Omega=0$.
\end{definition}

\noindent If we consider the trivial bundle $M\times G \rightarrow M$, there is a canonical flat connection, the Maurer-Cartan connection, given by $\theta^{\text{MC}}_{(x,g)}=(L_{g^{-1}}\circ \pi_{2})_*$ where $\pi_2:M\times G\rightarrow G$ is projection onto the second coordinate, and $L_{g^{-1}}:G\rightarrow G$ is left multiplication by $g^{-1}$.

\begin{lemma}[Dupont Theorem 3.21 \cite{dupont}]
     A connection $\theta$ on principal bundle $E\rightarrow M$ is flat if and only if around every point in $M$ there is a neighborhood $U$ and a trivialization $\phi_{U}:E\vert_{U}\rightarrow U\times G$ so that $\theta\vert_{U}=(\phi_{U})^{*}\theta^{\text{MC}}$.
\end{lemma}

\begin{theorem}[Cartan, see Dupont Theorem 8.1 \cite{dupont}]\label{cartan}
    Let $I^{*}(G)$ denote the algebra of $G$-invariant homogeneous polynomials on $\mathfrak{g}$ as defined in Dupont (see \cite{dupont} pg. 62). If $G$ is a compact Lie group, then the Chern-Weil homomorphism $\omega:I^{*}(G)\rightarrow H^{*}(BG)$ given by $\omega(P)=[P(\Omega)]\in H^{2k}(BG)$ is a ring isomorphism. (Here we consider $BG$ as the geometric realization of the simplicial bar construction $NG$, which is a simplicial manifold.)
\end{theorem}

\begin{definition}
    If $X,M$ are smooth manifolds, $E\rightarrow X\times M$ a smooth principal $G$-bundle, we say that a connection on $E$ is fiberwise flat if its restriction over $\{x\}\times M$ is a flat connection for each $x\in X$.
\end{definition}

\begin{definition}
    Given a flat principal $G$-bundle $E\xrightarrow{\pi} M$ with flat connection $\theta$ and basepoint $m_0\in M$, the holonomy representation at $x\in \pi^{-1}(m_0)$ is the map \[\pi_{1}(M,m_0) \xrightarrow{\mathcal{H}} G\] defined by $[\gamma]\rightarrow g$, where $g$ is the unique element satisfying $\widetilde{\gamma}(1)\cdot g= \widetilde{\gamma}(0)$ for $\widetilde{\gamma}$ the unique horizontal lift with respect to $\theta$ of $\gamma$ so that $\widetilde{\gamma}(0)=x$. 
\end{definition}

\begin{definition}\label{lbl1}
    If $\Gamma$ is a discrete group and $G$ a Lie group, an $X$-family of representations is a continuous map $\rho:X\rightarrow\mathrm{Hom}(\Gamma, G)$, where $\mathrm{Hom}(\Gamma, G)$ has the topology inherited as a subspace of $G^\Gamma$. If $X$ is a smooth manifold, an $X$-family is called smooth if for each $\gamma\in \Gamma$ the map $X\rightarrow G$ given by $x\mapsto \rho_{x}(\gamma):=\rho(x)(\gamma)$ is smooth. 
\end{definition}

Let $M$ be a connected manifold with universal cover $\widetilde{M}$. We can view $\widetilde{M}$ as the set of based homotopy classes of paths $l:[0,1]\rightarrow M$ with $l(1)=m_0$. If $\widetilde{m_{0}}$ is the constant path at $m_0$, then the map $q:(\widetilde{M},\widetilde{m_{0}})\rightarrow (M,m_0)$ given by evaluation at 0 is a right principal $\pi_{1}(M,m_0)$-bundle, with action given by $[l]\cdot [\alpha]=[l\cdot \alpha]$ where $l\cdot \alpha$ is the loop given by tracing out $l$ and then $\alpha$. \\

Now let $\rho$ be an $X$-family of representations of $\pi_{1}(M,m_0)$. We can construct a right principal $G$-bundle \[E_{\rho}(M)=X\times \widetilde{M}\times G/\pi_{1}(M,m_0)\xrightarrow{\pi} X\times M\] where we quotient by the $\pi_{1}$-action given by $(x,\widetilde{m},g)\cdot \alpha=(x,\widetilde{m}\cdot \alpha, \rho_{x}(\alpha)^{-1}g)$ and define $\pi([x,\widetilde{m},g])=(x,q(\widetilde{m}))$. If $\pi:X\times \widetilde{M}\times G\rightarrow E_{\rho}(M)$ is the quotient by the $\pi_{1}(M,m_0)$-action, then as the $\pi_{1}(M,m_0)$-action is smooth, free, and properly discontinuous there is a standard smooth structure on $E_\rho$ so that $\pi$ is a normal covering map (\cite{lee} Theorem 21.13). Throughout this paper we assume $E_\rho$ has this smooth structure. Some basic properties of this bundle construction are given in Baird and Ramras \cite{bairdramras}. \\ 

Note that if $X$ is a point, we get a flat bundle. By Theorem \ref{cartan} we see that flat bundles have trivial rational characteristic classes. For a more general algebraic proof see \cite{kambertondeur}.\\

We would now like to prove the existence of a fiberwise flat connection on bundles of the type $E_\rho$. First, we prove the following lemma.

\begin{lemma}\label{lbl3}
Let $K$ be a discrete group, $G$ a Lie group, $\pi:E\rightarrow B$ a smooth principal $G$-bundle. Assume $E$ and $B$ have a $K$-action that is free, commutes with the $G$-action on $E$, and induces covering maps $\psi:E\rightarrow E/K$ and $\bar{\psi}:B\rightarrow B/K$. Assume also that $\pi$ is $K$-equivariant. If $E\rightarrow B$ has a $K$-invariant connection $\theta$, so $\theta$ satisfies $\theta_{x}(w)=\theta_{x\cdot k}((\cdot k)_{*}w)$ for all $x\in E$, $w\in T_{x}(E)$, then $\theta$ descends to a connection on the $G$-bundle $\Bar{\pi}:E/K \rightarrow B/K$ with the $G$-action descended from $E$. If we further assume $F\rightarrow B$ is another $G$-bundle with such a $K$-action and $\phi:E\rightarrow F$ is a $G$-equivariant map that is also $K$-equivariant, so $\phi$ descends to a $G$-equivariant map $\overline{\phi}:E/K\rightarrow F/K$, then if a $K$-invariant connection on $E$ is the pullback by $\phi$ of a $K$-invariant connection on $F$, then the connection on $E/K$ is the pullback by $\overline{\phi}$ of the connection on $F/K$.
\end{lemma}

\begin{proof}
    First note that the map $\Bar{\pi}:E/K\rightarrow B/K$ is a principal $G$-bundle. To see this, note the $K$-action commutes with the $G$-action, so the $G$-action on $E$ descends to a well-defined $G$-action on the quotient space. For the required locally trivial neighborhoods, consider an open neighborhood $U_{\alpha} \subset B$ on which $\overline{\psi}$ is a local homeomorphism as a covering map. Note then that $\overline{\psi}(U_\alpha)$ is evenly covered. Further restrict $U_{\alpha}$ to a neighborhood $V_{\alpha}$ over which $E$ is trivial. One can then check that $\overline{\psi}(V_\alpha)$ is a locally trivial neighborhood of $\overline{\pi}$ and every point in $B/K$ is contained in such a neighborhood.\\

Let $\theta$ be a $K$-invariant connection on $E\rightarrow B$, and let $[x]=\psi(x)$. Define a connection $\Bar{\theta}$ on $E/K$ by $\Bar{\theta}_{[x]}(\psi_{*}w):=\theta_{x}(w)$ for each $x\in E$, noting that $\psi$ is a local diffeomorphism so every vector in $T_{[x]}(E/K)$ can be given as $\psi_{*}w$ for some $w\in T_{x}E$. This is well-defined as we assume $\theta$ is $K$-invariant. To see that the first condition of Definition \ref{def1} is satisfied, note that because the $K$-action on $E$ commutes with the $G$-action 
    \begin{equation}\label{eq1}
    v_{x\cdot k}=(g\mapsto (x\cdot k)\cdot g)_{*}=(\cdot k)_{*} \circ (g\mapsto x\cdot g)_{*}=(\cdot k)_{*} \circ v_{x}.
    \end{equation}
    Because the $G$-action on $E/K$ is defined as $[x]\cdot g=[x\cdot g]$, then $v_{[x]}:\mathfrak{g}\rightarrow T_{[x]}(E/K)$ is equal to $\psi_{*}(v_x)$ so $\Bar{\theta}_{[x]}(v_{[x]})=\theta_{x}(v_x)=id$. $\Bar{\theta}$ satisfies the second condition of Definition \ref{def1} as the $K$ and $G$ actions on $E$ commute so by definition of the $G$-action on $E/K$ the differentials $(R_g)_{*}$ and $\psi_{*}$ commute, so
\begin{equation*}
\begin{split}
    R_g^{*}\Bar{\theta}_{[x]}(\psi_{*}w)  & =\Bar{\theta}_{[x\cdot g]}((R_g)_{*}\psi_{*}w)  =\Bar{\theta}_{[x\cdot g]}(\psi_{*}(R_g)_{*}w)  =\theta_{x\cdot g}((R_g)_{*}w)\\  & =Ad(g^{-1})\circ \theta_{x}(w) =Ad(g^{-1})\circ \Bar{\theta}_{[x]}(\psi_{*}w).
\end{split}
\end{equation*}
 
 Now assume we are additionally given the $G$-bundle $F\rightarrow B$ with connections as described above. Let $\alpha$ and $\beta$ be connections on $E$ and $F$ respectively, where $\alpha=\phi^{*}\beta$, and let $\psi_E:E\rightarrow E/K$ and $\psi_F:F\rightarrow F/K$ be the respective quotients by $K$. Then
\begin{equation*}
\begin{split}
(\overline{\phi})^{*}\overline{\beta}_{[x]}((\psi_E)_{*}v) & =\overline{\beta}_{\overline{\phi}[x]}(\overline{\phi}_{*}((\psi_E)_{*}v))=\overline{\beta}_{\overline{\phi}[x]}((\psi_F)_{*}(\phi_{*}v))=\beta_{\phi(x)}(\phi_{*}v)\\
     &=\alpha_{x}(v)=\overline{\alpha}_{[x]}((\psi_{E})_{*}v).
\end{split}
\end{equation*}
\end{proof}

\begin{remark}\label{rmk}
    If the connection $\theta$ on $E\rightarrow B$ in Lemma \ref{lbl3} is flat, then by definition the descended connection on $E/K$ is flat.
\end{remark}
\begin{proposition}\label{lbl2}
    If $X$ and $M$ are smooth manifolds, $G$ a Lie group, and \\
    $\rho:X\rightarrow \mathrm{Hom}(\pi_{1}(M,m_0),G)$ is a smooth family of representations, then the bundle $E_{\rho}(M)$ admits a fiberwise flat connection whose holonomy over $\{x\}\times M$, computed at $[x,\widetilde{m_0},e]\in E_{\rho}(M)$, is $\rho(x).$
\end{proposition}

\begin{proof}
    Let $U\subseteq M$ be an open neighborhood over which the universal cover $\widetilde{M}\xrightarrow{q}M$ is trivial. Denote a trivialization by $\widetilde{M}\rvert_{U}\xrightarrow{\phi_{U}}U\times \pi_{1}(M,m_0)$, where $\phi_U$ is equivariant with respect to the right multiplication $\pi_{1}(M,m_0)$-action. Then there is a $G$-bundle isomorphism \[E_{\rho}(M)\rvert_{X\times U}\xrightarrow{\widetilde{\phi}_{U}} X\times U\times G\] given by \[[x,\widetilde{m},g]\mapsto (x,\phi_{U}(\widetilde{m})_1, \rho_{x}(\phi_{U}(\widetilde{m})_2)g),\] where $\phi_{U}(\widetilde{m})_1$ and $\phi_{U}(\widetilde{m})_2$ denote the projections on the first and second coordinates respectively. Note that $\phi_{U}(\widetilde{m})_1=q(\widetilde{m})$. We then define a map $T(E_{\rho}(M)\rvert_{X\times U})\xrightarrow{\psi_U} \mathfrak{g}$ by the composition

   \begin{multline*}
    T_{[x,\widetilde{m},g]}(E_{\rho}(M))\rvert_{X\times U} \xrightarrow{(\widetilde{\phi}_U)_*}T_{\widetilde{\phi}_{U}([x,\widetilde{m},g])}(X\times U\times G)\cong T_{x}X\times T_{q(\widetilde{m})}U\times T_{g}G \\
    \xrightarrow{(L_{g^{-1}})_*\circ \pi_3} T_{e}(G),
   \end{multline*}

 where $\pi_3$ is projection onto the third coordinate. Let $\{\eta_{i}\}$ be a partition of unity subordinate to a numerable covering of $M$ by such open sets $U_i\subseteq M$ over which $\tilde{M}$ is trivial. Define $\theta$ on $E_{\rho}(M)$ by 
\[\theta(\vec{v})=\sum_{i}\eta_{i}(q(\widetilde{m}))\psi_{U_i}(\vec{v})\] for a vector $\vec{v}\in T_{[x,\widetilde{m},g]}(E_{\rho}(M)).$ It is routine to check that this is a connection. We now show that this connection is flat when restricted to the sub-bundle $E_{\rho}(M)\rvert_{\{x\}\times M}$ for each $x\in X$. To do this, we will show that it takes the same values on $E_{\rho}(M)\rvert_{\{x\}\times M}$ as the connection defined by Lemma \ref{lbl3} descended from the Maurer-Cartan connection on the trivial bundle \[\{x\}\times \widetilde{M}\times G\rightarrow \{x\}\times \widetilde{M},\] where we give $\{x\}\times \widetilde{M}\times G$ the $\pi_{1}(M,m_0)$-action defined by \[[x,\widetilde{m},g]\cdot \alpha=[x,\widetilde{m}\cdot \alpha, \rho_{x}(\alpha)^{-1}g].\] By Remark \ref{rmk}, this will give fiberwise flatness and will make the holonomy calculation simpler.\\

First, note that under the maps $(\widetilde{\phi}_{U})_*$ tangent vectors to $E_{\rho}(M)\rvert_{\{x\}\times U}$ map to tangent vectors of the form $(0,\vec{v_1},\vec{v_2})\in TX\times TU \times TG$. We now want to show that if $m=q(\widetilde{m})\in U\cap V\subseteq M$, then $\psi_{U}=\psi_{V}$ on vectors in $T_{[x,\widetilde{m},g]}(E_{\rho}(M)\rvert_{\{x\}\times M}$. Over a point $x\in U\cap V$, $\phi_V\circ \phi_{U}^{-1}:U\cap V\times \pi_{1}(M,m_0)\rightarrow U\cap V\times \pi_{1}(M,m_0)$ is equivariant with respect to the $\pi_{1}(M,m_0)$-action by right multiplication, so $\phi_V\circ \phi_{U}^{-1}$ is equivalent to the left multiplication map $\alpha \cdot (x,\beta)=(x,\alpha \beta)\in \{x\}\times \pi_{1}(M,m_0)$ for some $\alpha\in \pi_{1}(M,m_0)$. As  $\phi_V\circ \phi_{U}^{-1}$ is a homeomorphism and $\pi_{1}(M,m_0)$ is discrete then this $\alpha$ is the same for every point in some open neighborhood of $x\in U\cap V$. If we then define a left $\pi_{1}(M,m_0)$-action on $X\times M\times G$ by $\alpha \cdot (x,m,g)=(x,m,\rho_{x}(\alpha)g)$ then locally we get $\widetilde{\phi}_V=\alpha\cdot \widetilde{\phi}_U$ as $\rho_x$ is a homomorphism. Therefore, if $\vec{v}\in T_{[x,\widetilde{m},g]}E_{\rho}(M)\rvert_{\{x\}\times M}$ then  \[\psi_{V}(\vec{v})=(L_{g^{-1}\rho_{x}(\phi_{U}(\widetilde{m})_2)^{-1}\rho_{x}(\alpha)^{-1}})_{*}\circ \pi_{3}\circ (\alpha \cdot)_*\circ (\widetilde{\phi}_{U})_{*}(\vec{v})\] and on vectors of the form $(0,\vec{v_1},\vec{v_2})\in TX\times TM\times TG$, \[\pi_3\circ (\alpha \cdot)_*=(L_{\rho_{x}(\alpha)})_*\circ \pi_3\] so the composition simplifies to show $\psi_V=\psi_U$. Thus on vectors in \\ $T_{[x,\widetilde{m},g]}E_{\rho}(M)\rvert_{\{x\}\times M}$ the connection $\theta$ is given by $\theta=\psi_U$ for any $U\subseteq M$ as described above containing $q(\widetilde{m})$. This argument also shows that for a given $U$, different choices of trivialization $\phi_U$ give the same connection. It is straightforward to check that this is the same connection as that defined by Lemma \ref{lbl3}, so in particular is flat. \\

Now to show that this connection has the desired holonomy over $\{x\}\times M$. If $\gamma$ is a loop in $\{x\}\times M$ based at $m_0$, the horizontal lifts of this loop to $E_{\rho}(M)$ are of the form $[x,\widetilde{\gamma},e]$ where $\widetilde{\gamma}$ is a lift of the loop in $\widetilde{M}$. This is because the lift of this loop to a horizontal path at $(x,\widetilde{m_0},e) \in \{x\}\times \widetilde{M}\times G$ is of the form $(x,\widetilde{\gamma},e)$. Then note \[[x,\widetilde{\gamma}(1),e]\cdot \rho_{x}([\gamma])=[x,\widetilde{\gamma}(1),\rho_{x}([\gamma])]=[x,\widetilde{\gamma}(1)\cdot [\gamma],e]=[x,\widetilde{\gamma}(0),e]\] so we get the desired holonomy.  
\end{proof}

\begin{definition}
    A space $Y$ has finite type if $H_{j}(Y;\bbZ)$ is finitely generated for all $j\geq 0$.
\end{definition}

We now wish to prove the main result of this section. This theorem is a generalization of Theorem 3.5 in Baird and Ramras \cite{bairdramras} to principal $G$-bundles of the type $E_{\rho}$ for $G$ a compact Lie group. 

\begin{theorem}\label{lbl4}
    Let $Z$ be a path connected topological space, $G$ a compact Lie group, and let $$\rho:X\rightarrow \mathrm{Hom}(\pi_{1}(Z,z_0),G)$$ be a family of representations parametrized by a space $X$ with $H^{j}(X;\bbQ)=0$ for $j>m$. Assume that either $X$ or $Z$ has finite type. Then for $i>0$, the image of the homomorphism \[H^{2m+2i}(BG;\bbQ)\rightarrow H^{2m+2i}(X\times Z;\bbQ)\] induced by the classifying map of $E_{\rho}(Z)$ is zero. 
\end{theorem}

The method of proof will be similar, so we make use of the following results proved by Baird and Ramras.

\begin{lemma}[Baird and Ramras Lemma 3.4]
    Let $Y$ be a topological space. Then there exists a set of smooth, closed manifolds $\{M_k\}_{k\in K}$, and a map $f: M:= \coprod_{k} M_k\rightarrow Y$ such that for each $j\geq 0$, $$f^{*}: H^{j}(Y;\bbQ)\rightarrow H^{j}(M;\bbQ)\cong \prod_{k}H^{j}(M_k;\bbQ)$$ is injective. If $H^{j}(Y;\bbQ)=0$ for $j>m$, we can assume dim$(M_k)\leq m$ for all $k$. 
\end{lemma}

\begin{lemma}[Baird and Ramras Theorem 2.1]\label{lbl5}
    Let $i:X\xhookrightarrow{}\bbR^{n}$ be a real algebraic set and let $M$ be a smooth manifold. Assume that either $X$ or $M$ is compact. Then for every continuous map $\phi:M\rightarrow X$ there exists a map $\phi ':M\rightarrow X$ homotopic to $\phi$ such that $i\circ \phi '$ is smooth. 
\end{lemma}

Lemma \ref{lbl5} allows us to replace a family of representations $\rho$ by a smooth family $\rho '$ homotopic to $\rho$. We note that a more general approach to such a smoothing problem is given in \cite{fernando-ghiloni} Theorem 1.4. Now we can apply Lemma \ref{lbl2} to $E_{\rho '}\cong E_{\rho}$. To do this we are left to show that for some finitely generated group $\Gamma$ and compact Lie group $G$, Hom$(\Gamma,G)$ is a real algebraic set, and that the smoothness of \ref{lbl5} gives smoothness as in Definition \ref{lbl1}. 

\begin{claim}
    Let $\Gamma$ be a finitely generated discrete group and $G$ a compact Lie group. Then $\mathrm{Hom}(\Gamma,G)$ is a real algebraic set. 
\end{claim}

\begin{proof}
    First we note that every compact Lie group is a real algebraic set. This follows because every compact Lie group has a faithful representation (see \cite{repscompactliegps} Ch.3 Theorem 4.1), so the image of this representation is a compact subgroup of $GL_{n}(\bbC)$ for some $n$. Then by Theorem 3.4.5 in Onishchik and Vinberg \cite{o-v} we get that the image is  real algebraic. To see that Hom$(\Gamma,G)$ is real algebraic, note it is a subset of $G^s$ under the restriction map Hom$(\Gamma,G)\rightarrow G^s$ given by $\phi \mapsto (\phi(g_1),\dots,\phi(g_s))$, where $\{g_1,\dots,g_s\}$ is the generating set of $\Gamma$. Elements in the image of the restriction map are $s$-tuples of elements in $G$ whose components satisfy relations given by passing the group relations on $\Gamma$ through some homomorphism to $G$. As $G^s$ is real algebraic, the embedding $G^s\subseteq \bbR^{2n^2}$ then defines Hom$(\Gamma,G)$ as a subset of $\bbR^{2n^2}$ given by points satisfying polynomial relations.
\end{proof}

Then, given some $X$-family of representations $\rho$ into Hom$(\Gamma,G)$, Lemma \ref{lbl5} lets us homotope it to a map $X\rightarrow$ Hom$(\Gamma,G)$ for which the composite \[X\rightarrow \text{Hom}(\Gamma,G)\rightarrow G^{s}\xhookrightarrow{}\bbR^{2n^2}\] is smooth. As closed subgroups of Lie groups can be given a compatible smooth structure making them embedded Lie groups, the embedding of $G^s$ as a real algebraic subset of $\bbR^{4n^2}$ is actually an embedding as a smooth submanifold, first by embedding as a Lie subgroup of $GL_{n}(\bbC)$ and then embedding this as a smooth submanifold of $\bbR^{4n^2}.$ Thus \[X\rightarrow \text{Hom}(\Gamma,G)\rightarrow G^{s}\xhookrightarrow{}\bbR^{2n^2}\] is smooth if and only if the composition $X\rightarrow G^s$ is smooth, and as $\Gamma$ is discrete this is equivalent to saying that $\rho$ is homotopic through $X$-families of representations to a smooth $X$-family.

\begin{proof}[Proof Of Theorem \ref{lbl4}]
First we reduce to the case in which $X$ and $Z$ are closed, connected, smooth manifolds, using the same argument as Baird and Ramras. In particular note that we have shown we can assume without loss of generality that $\rho$ is a smooth $X$-family.\\

Lemma \ref{lbl2} then guarantees the existence of a fiberwise flat connection $\theta$ on $E_{\rho}(Z)$. Choose local coordinates $\{x_1,\dots,x_m,z_1,\dots, z_p\}$ of $X\times Z$ so that each slice $\{x\}\times Z$ is given locally by setting the $x_i$ coordinates as constants. Let $\{\gamma_1,\dots,\gamma_k\}$ be a basis of $\mathfrak{g}$. Then the connection $\theta$ can be given locally by $\theta=\omega_1 \gamma_1+\dots+\omega_k \gamma_k$ where each $\omega_i$ is an $\bbR$-valued 1-form on $E_\rho(Z)$, which are locally given in the span of $\{dx_1,\dots,dx_m,dz_1,\dots, dz_p\}$. Then locally \[\theta\wedge\theta=\sum_{1\leq i\leq j \leq k}(\omega_i \wedge \omega_j)\gamma_i \otimes \gamma_j\] so \[\Omega=d(\omega_1 \gamma_1+\dots \omega_k \gamma_k)+\frac{1}{2}\sum_{1\leq i\leq j \leq k}(\omega_i \wedge \omega_j)[\gamma_i, \gamma_j]\] Fiberwise flatness then gives that the coefficients of $\Omega$ are in the span of the 2-forms of the form $dx_i\wedge dx_j$ and $dx_i\wedge dz_j$. Then any Chern-Weil form given by a homogeneous polynomial of degree $>m$ in the curvature must be 0. By a theorem of Cartan (see \cite{dupont} Theorem 8.1) every characteristic class of $E_{\rho}(Z)$ is given by the de Rham cohomology class of such Chern-Weil forms, so every characteristic class in $H^{2m+2i}(X\times Z,\bbR)$ is zero for $i>0$. It follows immediately that this is also true for rational coefficients. 

\end{proof}

\begin{remark}
If one considers the foliation of $X\times M$ by leaves $\{x\}\times M$, then the horizontal vectors of the fiberwise flat connection on $E_\rho \rightarrow X\times M$ constructed in Lemma \ref{lbl2} define a flat partial connection as defined in \cite{kambertondeurfoliated}. The proof of Theorem \ref{lbl4} can then be seen to be a special case of \cite{kambertondeurfoliated} Corollary 4.30.
\end{remark}
\section{Connectivity of the Space of Flat Unitary Connections}

Let $M$ be a closed, connected, smooth manifold with fundamental group $\Gamma=\pi_1(M,m_0)$. Let $f:M\rightarrow B\Gamma$ be a map classifying the universal cover $\widetilde{M}\rightarrow M$ that takes the basepoint $m_0$ to the canonical basepoint of $B\Gamma$. Then we can consider the composition \[\text{Hom}(\Gamma,G)\xrightarrow{B}\text{Map}_{*}(B\Gamma,BG)\xrightarrow{f^{*}}\text{Map}_{*}(M,BG).\]
We want to apply the characteristic class restrictions of Section 2 to study this composite.

 By studying the first map in this composition, Baird and Ramras \cite{bairdramras} showed that for $U(n)$-bundles over closed, smooth, aspherical manifolds $\mathcal{A}_{\mathrm{flat}}(E)$ is highly disconnected. For example, the space of flat connections for a flat $U(n)$-bundle, $n\geq 2$, over an orientable aspherical 3-manifold has infinitely many path components. This is in contrast to work by Ramras \cite{ramras-yangmills} that showed the space of flat connections $\mathcal{A}_{\mathrm{flat}}(E)$ for a principal $U(n)$-bundle over a surface is highly connected, depending on the genus of the surface and rank of the bundle.\\

We will show that expanding the arguments of Baird and Ramras allows us to drop the asphericity condition on the base manifold, and we gain similar results for other classical Lie groups. \\

First we note that the map $f^{*}\circ B$ is closely related to the $E_\rho$ bundles that we explored earlier:

\begin{proposition}\label{lbl6}
    Let $G$ be a Lie group, $X$ be any space. Then any $G$-bundle of the form $E_{\rho}\rightarrow X\times M$ is classified by the composition $$X\times M \xrightarrow{\rho \times f}\mathrm{Hom}(\Gamma,G)\times B\Gamma\xrightarrow{B\times id}\mathrm{Map}_{*}(B\Gamma,BG)\times B\Gamma\xrightarrow{ev} BG$$ where $\Gamma=\pi_{1}(M,m_0)$ and $f:M\rightarrow B\Gamma$ classifies the universal cover. 
\end{proposition}

\begin{proof}
   This proof is similar to that of Lemma 4.1 in Baird and Ramras \cite{bairdramras}.
\end{proof}

The main theorem for the section is the following:

\begin{theorem}\label{lbl7}
Let $M$ be a closed, connected, smooth manifold of dimension d with fundamental group $\Gamma=\pi_{1}(M,m_0)$. Let $f:M\rightarrow B\Gamma$ be a map classifying the universal cover $\widetilde{M}\rightarrow M$ that takes the basepoint $m_0$ to the canonical basepoint of $B\Gamma$. Let $\beta_{k}(M)=\mathrm{Rank}(H^{k}(M;\bbQ))$. Then for each $\rho\in \mathrm{Hom}(\Gamma,\text{GL}_{n}(\bbC))$ and each $m\geq 0$, the induced map on homotopy groups $$(f^{*}\circ B)_{*}:\pi_{m}(\mathrm{Hom}(\Gamma,\text{GL}_{n}(\bbC)),\rho)\rightarrow \pi_{m}(\mathrm{Map}_{*}(M,B\text{GL}_{n}(\bbC)),B\rho \circ f)$$ satisfies $\mathrm{Rank}(\mathrm{coker}((f^{*}\circ B)_{*}))\geq \sum_{i=1}^{\infty}\beta_{2i+m}(M)$ when $n\geq (m+d)/2$.\\
\end{theorem}

The proof is analogous to that of Theorem 1.2 in Baird and Ramras \cite{bairdramras}, so we do not give the details. The idea behind the proof is that there are cohomology classes which we know by Theorem \ref{lbl4} are not characteristic classes of $E_\rho$ bundles. If we can find bundles with those characteristic classes, then we can build up a non-trivial cokernel of the map $(f^{*}\circ B)_{*}$. To do this, we need the following lemma from Baird and Ramras. A detailed proof for the $O(n)$ case can be found in the next section.

\begin{lemma}[Baird and Ramras \cite{bairdramras} Lemma 4.2]\label{lbl8}
    Let $X$ be a finite CW complex. Let $c_i$ denote the $i$th Chern class in rational cohomology. Given $m>0$ and $x\in H^{2i}(S^m \wedge X;\bbQ)$, there exists a class $\phi \in K^{0}(S^{m}\wedge X)$ such that $c_{i}(\phi)=qx$ for some non-zero rational number $q\in \bbQ^{*}$, and $c_{j}(\phi)=0$ for $j\neq i$. 
\end{lemma}

We now want to relate Theorem \ref{lbl7} to the space of flat connections for a unitary bundle (or more generally a principal $G$-bundle for $G$ a compact Lie group). Let $M$ be a closed, connected, smooth $d$-dimensional manifold. Let $E\rightarrow M$ be a principal $G$-bundle over $M$ with chosen basepoint $e_0$ that projects to chosen basepoint $m_0$. We consider connections on $E$ of Sobolev class $L_{k}^{p}$ for some fixed constants $p\in \bbR$, $k\in \bbN$ with min$(d/2,4/3)<p<\infty,k\geq 2,$ and $kp>d$. We also consider gauge transformations of $E$ of Sobolev class $L^{p}_{k+1}$. The details behind these restrictions on the constants can be found in Ramras \cite{ramras-yangmills} Section 3.2. They allow us to say that the based gauge group $\mathcal{G}_{0}(E)$ acts continuously on the space of flat connections $\mathcal{A}_{\mathrm{flat}}(E)$, with $\mathcal{A}_{\mathrm{flat}}(E)/\mathcal{G}_{0}(E)\cong \text{Hom}(\pi_{1}(M),G)$. The quotient can be identified with the holonomy representation $\mathcal{H}$ sending a flat connection to the holonomy representation of that connection at basepoint $e_0$.

\begin{proof}[Proof of Theorem \ref{lbl9}]
    Let $\Gamma=\pi_{1}(M,m_0)$. We will consider the following commutative diagram:

\begin{equation}\label{dgrm1}
\xymatrix{\mathcal{A}_{\mathrm{flat}}(E)  \ar[rr]^-{\mathcal{T}} \ar[d]^{\mathcal{H}}&& \text{Map}_{*}^{G}(E,EG)\ar[d]^{q} \\
\text{Hom}^{E}(\Gamma,G) \ar[r]^-{B}& \text{Map}_{*}^{E}(B\Gamma,BG) \ar[r]^-{f^*}& \text{Map}_{*}^{E}(M,BG) \\}
\end{equation}

Here $\mathcal{H}$ is the holonomy representation, $\text{Map}_{*}^{E}(M,BG)$ is the path component of the based mapping space consisting of maps classifying $E$, $\text{Hom}^{E}(\Gamma,G)$ is the preimage of $f^{*}\circ B$, and $\text{Map}_{*}^{E}(B\Gamma,BG)$ is the preimage of $f^{*}$. To justify the composition $f^{*}\circ B$, note that if we are given a flat connection $A\in \mathcal{A}_{\mathrm{flat}}(E)$ then $\mathcal{H}(A):\Gamma \rightarrow G$ (computed at chosen basepoint $e_0$ over $m_0$) induces a $G$-equivariant map $E\Gamma \xrightarrow{E(\mathcal{H}(A))} EG$ between simplicial models that covers the functorial map $B\Gamma \xrightarrow{B(\mathcal{H}(A))} BG$. Then we get a pullback diagram 

\begin{equation*}
\xymatrix{E_{\mathcal{H}(A)}(M)  \ar[r]^-{\tilde{f}_{\mathcal{H}(A)}} \ar[d]& E_{\mathcal{H}(A)}(B\Gamma) \ar[r]^-{u_{\mathcal{H}(A)}} \ar[d]& EG\ar[d] \\
M \ar[r]^-{f}& B\Gamma \ar[r]^-{B(\mathcal{H}(A))}& BG \\}
\end{equation*}

where $\tilde{f}_{\mathcal{H}(A)}([\tilde{m},g])=[\tilde{f}(\tilde{m}),g]$ and $u_{\mathcal{H}(A)}([e,g])=E(\mathcal{H}(A))(e)\cdot g$. As shown in Lemma \ref{lbl2} there is a canonical flat connection $A_{\mathcal{H}(A)}$ on $E_{\mathcal{H}(A)}(M)$ whose holonomy at the chosen basepoint $[\tilde{m}_0,e]$ is $\mathcal{H}(A)$. There exists a canonical isomorphism of $G$-bundles $E\xrightarrow{\phi_{A}}E_{\mathcal{H}(A)}(M)$ so that $\phi_{A}(e_0)=[\tilde{m}_0,e]$ and $\phi_{A}^{*}(A_{\mathcal{H}(A)})=A$ (see Ramras \cite{ramras-stablemoduli} Prop 8.4 for details on this isomorphism). We then define $\mathcal{T}$ as the composition $\mathcal{T}(A)=u_{\mathcal{H}(A)}\circ \tilde{f}_{\mathcal{H}(A)}\circ \phi_{A}$. 
The continuity of $\mathcal{T}$ follows as in Ramras \cite{ramras-stablemoduli}, and we have shown that diagram \ref{dgrm1} commutes.\\

Now consider the commutative diagram:

\begin{equation*}
\xymatrix{\text{Hom}^{E}(\Gamma,G)  \ar[r]^-{f^{*}\circ B} \ar[d]^{\text{Id}} & \text{Map}_{*}^{E}(M,BG) \ar[d]^{\text{Id}} & \ar[l] \{B(\mathcal{H}(A))\circ f\} \ar[d]^{i} \\
\text{Hom}^{E}(\Gamma,G)\ar[r]^-{f^{*}\circ B}& \text{Map}_{*}^{E}(M,BG) &\ar[l]_{q} \text{Map}_{*}^{G}(E,EG) \\}
\end{equation*}

Here $i$ sends $B(\mathcal{H}(A))\circ f$ to $u_{\mathcal{H}(A)}\circ \tilde{f}_{\mathcal{H}(A)}$. By Gottlieb \cite{gottlieb} Theorem 5.6, $\text{Map}_{*}^{G}(E,EG)$ is weakly contractible, so the vertical maps in this diagram are weak equivalences. Then as shown in Baird and Ramras Proposition 5.4 \cite{bairdramras}, this means the homotopy pullbacks of the rows are weak equivalent. The homotopy pullback of the top row is $\text{hofib}_{B(\mathcal{H}(A))\circ f}(f^{*}\circ B)$, and the bottom row is diagram \ref{dgrm1} without the top left corner. Then it suffices to show that diagram \ref{dgrm1} is homotopy cartesian. This follows from an argument identical to that given at the end of the proof of Theorem 5.5 in Baird and Ramras \cite{bairdramras}. 
    
\end{proof}

\begin{remark}
    As noted in Baird and Ramras, Lemma \ref{lbl9} also shows that the weak homotopy type of the space of flat connections is independent of choice of Sobolev class $L^{p}_{k}$ (as long as that class satisfies the requirements on $p$ and $k$ as described previously).
\end{remark}

Theorem \ref{lbl9} leads to the following corollary about the connectivity of the space of flat connections, which is a direct generalization of Corollary 1.3 in Baird and Ramras \cite{bairdramras}.

\begin{corollary}\label{lbl10}
    Let M be a closed, smooth, connected manifold of dimension d, and let $E\rightarrow M$ be a flat principal $U(n)$-bundle over M. Then for each $A_0\in \mathcal{A}_{\mathrm{flat}}(E)$ and each $m$ such that $0< m\leq 2n-d-1$, we have $$\mathrm{Rank}(\pi_{m}(\mathcal{A}_{\mathrm{flat}}(E),A_0))\geq \sum_{i=1}^{\infty}\beta_{m+2i+1}(M).$$ In addition, if $\sum_{i=1}^{\infty}\beta_{2i+1}>0$ and $2n-d-1\geq 0$, then $\mathcal{A}_{\mathrm{flat}}(E)$ has infinitely many path components. 
\end{corollary}

\begin{proof}
    The proof is identical to that of Corollary 1.3 in Baird and Ramras \cite{bairdramras}, instead using the long exact sequence in homotopy groups for the homotopy fiber sequence $$\mathcal{A}_{\mathrm{flat}}(E)\rightarrow \text{Hom}^{E}(\Gamma,U(n))\xrightarrow{f^{*}\circ B} \text{Map}^{E}_{*}(M,BU(n))$$ established in Theorem \ref{lbl9}.
\end{proof}

For example, the space of flat connections for a flat $U(n)$-bundle over $S^{2k+1}, k\geq 1$ has infinitely many path components if $n\geq k$.

\begin{remark}
    Diagram (\ref{dgrm1}) allows for comparison between the long exact sequences in homotopy for the holonomy fibration and the homotopy fiber sequence of $f^{*}\circ B$. Choose a basepoint $A_0\in \flatc(E)$.\\
    First, note that Diagram (\ref{dgrm1}) induces a commuting diagram between the long exact sequences of the fibration $\mathcal{H}$ and the fibration $q$, through which one can identify the boundary map $\partial_{\mathcal{H}}$ of the holonomy fibration with the map $(f^{*}\circ B)_*$ through the composition
    \begin{align*}
        \pi_{k}(\mathrm{Hom}^{E}(\Gamma,G))\xrightarrow{(f^{*}\circ B)_*}\pi_{k}(\mathrm{Map}^{E}_{*}(M,BG))\xrightarrow[\cong]{\partial_{q}}\pi_{k-1}(\mathrm{Map}_{*}^{G}(E,G))\\
        \xrightarrow[\cong]{(i_*)^{-1}}\pi_{k-1}(\mathcal{G}_0), 
    \end{align*} where $i$ is the inclusion of the Sobolev gauge group into the continuous gauge group, which is a weak equivalence (\cite{palais} Theorem 13.14).\\
    
   One can also show that the classes in $\pi_{m}(\mathcal{A}_{\mathrm{flat}}(E))$ that are constructed in the proof of Corollary \ref{lbl10} can be identified as classes coming from $\pi_{m}(\mathcal{G}_0(E)).$ To see this we first construct a long exact sequence of homotopy groups \begin{align*}
       \dots \rightarrow \pi_{k}(\mathrm{Map}_{*}^{E}(M,BG))\xrightarrow{(\phi)_*\circ (i_*)^{-1}\circ \partial_q}\pi_{k-1}(\flatc(E))\\
       \xrightarrow{\mathcal{H}_*} \pi_{k-1}(\mathrm{Hom}^{E}(\Gamma,G))\xrightarrow{(f^{*}\circ B)_*}\pi_{k-1}(\mathrm{Map}_{*}^{E}(M,BG))\rightarrow \dots
   \end{align*} by applying the construction in  \cite{Hatcher} Section 2.2, Exercise 38 to the commuting diagram of long exact sequences for the fibrations $\mathcal{H}$ and $q$, recalling that $\mathrm{Map}_{*}^{G}(E,EG)$ is weakly contractible. Here $\mathcal{G}_0\xrightarrow{\phi}\flatc(E)$ is the inclusion of the fiber of $\mathcal{H}$ and $\partial_q$ is the boundary map for the long exact sequence of $q$. By comparing with the long exact sequence in homotopy for the homotopy fiber sequence \begin{equation}\label{fiberseq}
       \mathcal{A}_{\mathrm{flat}}(E)\rightarrow \text{Hom}^{E}(\Gamma,G) \xrightarrow{f^{*}\circ B} \text{Map}^{E}_{*}(M,BG) \end{equation} established in the proof of Theorem \ref{lbl9}, we see that classes in the image of the boundary map for (\ref{fiberseq}) are in the image of $(\phi)_*\circ (i_*)^{-1}\circ \partial_q$, so in particular are images of classes under the map $\pi_{*}(\mathcal{G}_0)\xrightarrow{\phi_*}\pi_{*}(\flatc(E)).$
\end{remark}

\section{The Case of $O(n)$-bundles}

The methods used above extend to the case of principal $O(n)$-bundles and give connectivity results for the space of flat connections for a real vector bundle. \\

Let $p_i(E)\in H^{4i}(E;\bbZ)$ denote the $i$th Pontryagin class of a real vector bundle and $\mathrm{KO}(X)$ denote the real K-theory of a CW-complex $X$. 

\begin{lemma}\label{lbl11}
    Let $X$ be a finite CW complex. Given $m>0$ and $x\in H^{4i}(S^m \wedge X;\bbQ)$, there exists a class $\phi\in \mathrm{KO}^{0}(S^m \wedge X)$ such that $p_i(\phi)=qx$ for some non-zero rational number $q\in \bbQ^{*}$, and $p_j(\phi)=0$ for $j\neq i$.
\end{lemma}

To prove this we will need the following result. Proofs of this result are sketched in many sources (see \cite{karoubi}). A detailed proof can be found in the Appendix. 

\begin{proposition}\label{lbl12}
    For $X$ a finite CW complex, there is a ring isomorphism $$\mathrm{KO}^{0}(X)\otimes \bbQ \cong \prod_{i\geq 0}H^{4i}(X;\bbQ)$$ induced by the complexification map $\mathrm{KO}^0(X)\rightarrow \mathrm{K}^0(X)$ followed by the Chern character isomorphism. 
\end{proposition}

\begin{proof}[Proof of Lemma \ref{lbl11}]
By Proposition \ref{lbl12} there exists a class\\$\phi\in KO^{0}(S^m\wedge X)$ so that $Ch(c(\phi))=qx$ for some $q\in \bbQ^{*}$, where \[c: \mathrm{KO}^{0}(X)\rightarrow \mathrm{K}^{0}(X)\] is the complexification. Note $Ch$ is a polynomial in the Chern classes, but all cup products in $H^{*}(S^m\wedge X;\bbQ)$ are zero. Then \[Ch(\phi)=q_{1}c_1(\phi)+\dots+q_{k}c_{k}(\phi)+\dots\] so the desired result follows as $x$ was chosen to be in a desired dimension. 
\end{proof}

\begin{theorem}\label{lbl16}
    Let $M$ be a closed, connected, smooth manifold of dimension d with fundamental group $\Gamma=\pi_{1}(M,m_0)$. Let $f:M\rightarrow B\Gamma$ be a map classifying the universal cover $\widetilde{M}\rightarrow M$ that takes the basepoint $m_0$ to the canonical basepoint of $B\Gamma$. Let $\beta_{k}(M)=\mathrm{Rank}(H^{k}(M;\bbQ))$. Then for each $\rho\in \mathrm{Hom}(\Gamma,\text{GL}_{n}(\bbR))$ and each $m\geq 0$, the induced map on homotopy groups $$(f^{*}\circ B)_{*}:\pi_{m}(\mathrm{Hom}(\Gamma,\text{GL}_{n}(\bbR)),\rho)\rightarrow \pi_{m}(\mathrm{Map}_{*}(M,B\text{GL}_{n}(\bbR)),B\rho \circ f)$$ satisfies $\mathrm{Rank}(\mathrm{coker}((f^{*}\circ B)_{*}))\geq \sum_{i=1}^{\infty}\beta_{3m+4i}(M)$ when $n\geq m+d+1$.
\end{theorem}

\begin{proof}
This proof is analogous to the proof of Theorem 1.2 in Baird and Ramras \cite{bairdramras}, with the following adjustments for real K-theory. First note that the map $B$GL$_n(\bbR)\rightarrow B$GL$_{n+1}(\bbR)$ induced by the standard inclusion is $n$-connected as we now consider the Serre fibration $O(n)\rightarrow O(n+1)\rightarrow S^n$. We also notice a change in the Betti numbers bounding\\
$\mathrm{Rank}(\mathrm{coker}((f^{*}\circ B)_{*}))$. This lower bound comes from the rank of a free abelian subgroup in $\widetilde{\mathrm{KO}}^{0}(S^m\wedge M)$ that is not in the image of $(f^{*}\circ B)_*$. We define this subgroup by associating cohomology classes with real K-theory classes using Proposition \ref{lbl12}. As this proposition gives cohomology classes in $\prod_{i\geq 0}H^{4i}(X;\bbQ)$, we can choose a $\bbQ$-basis \[\{a_{ij}\}_{j}\subset H^{3m+4i}(M;\bbQ)\cong H^{4(m+i)}(S^m\wedge M;\bbQ)\] that are realized as characteristic classes of bundles in a subgroup of \\ $\widetilde{\mathrm{KO}}^{0}(S^m\wedge M)$. We then show as in Baird and Ramras that by construction this subgroup is not in the image of $(f^{*}\circ B)_*$.
\end{proof}

\begin{corollary}\label{lbl17}
     Let M be closed, smooth, connected manifold of dimension d, and let $E\rightarrow M$ be a flat principal $O(n)$-bundle over M. Then for each $A_0\in \mathcal{A}_{\mathrm{flat}}(E)$ and each $m$ such that $0< m\leq n-d-2$, we have $$\mathrm{Rank}(\pi_{m}(\mathcal{A}_{\mathrm{flat}}(E),A_0))\geq \sum_{i=1}^{\infty}\beta_{3m+4i+1}(M).$$ In addition, if $\sum_{i=1}^{\infty}\beta_{4i+1}>0$ and $n-d-2\geq 0$, then $\mathcal{A}_{\mathrm{flat}}(E)$ has infinitely many path components. 
\end{corollary}

\begin{proof}
    The proof is identical to that of Corollary \ref{lbl10}.
\end{proof}

\begin{example}
 The space of flat connections of any flat $O(n)$-bundle over $\bbR \amsbb{P}^{4k+1}$ with $n\geq 4k+3$, $k>0$ has infinitely many path components. 
\end{example}

\section{Results for $SU(n)$, $SO(n)$, and Spin($n$)}
The results for $U(n)$ and $O(n)$ can be extended to give results for $SU(n)$, $SO(n)$, and Spin($n$) by reducing structure group for the bundles appearing Lemmas \ref{lbl8} and \ref{lbl11}.  Recall the following definitions:\\

\begin{definition}
    Let $E\rightarrow B$ be a principal $G$-bundle, and $H$ a subgroup of $G$. We say the structure group of $E$ can be reduced to $H$ if there exists a principal $H$-bundle $F$ and an isomorphism \[E\cong F\times_{H}G:= (F\times G)/H\] where $H$ acts by $h(f,g)=(fh^{-1},hg)$.
\end{definition}

\begin{corollary}[\cite{maycharclasses} Corollary 3.1]\label{lbl14}
    The structure group of a complex vector bundle can be reduced to $SU(n)$ if and only if the first Chern class is zero.  
\end{corollary}

It is a standard fact that $H^{*}(BSU(n);\bbQ)\cong \bbQ[d_2,\dots,d_n]$, where $d_{i}(ESU(n))$ is the pullback of the Chern class $c_{i}(EU(n))\in H^{2i}(BU(n);\bbQ)$ under the map induced by the inclusion $SU(n)\hookrightarrow U(n)$. For an $SU(n)$-bundle $E\rightarrow X$, define $d_i(E)\in H^{2i}(X;\bbQ)$ to be the pullback of $d_i(ESU(n))$ by the classifying map of $E$.

\begin{lemma}\label{lbl15}
    Let $X$ be a finite CW complex. Then given $m>0$ and $x\in H^{2i}(S^{m}\wedge X;\bbQ)$, there exists an $SU(n)$-bundle $\phi$ over $S^{m}\wedge X$ with $d_i(\phi)=qx$ for some non-zero $q\in \bbQ$ and $d_j(\phi)=0$ for $j\neq i$.
\end{lemma}

\begin{proof}
    Baird and Ramras Lemma 4.2 \cite{bairdramras} construct complex vector bundles that satisfy such properties for their Chern classes. Consider such a vector bundle whose first Chern class is zero. By Corollary \ref{lbl14} the classifying map of this bundle factors through $BSU(n)$. The result then follows by the properties of the Chern classes for our chosen bundle. 
\end{proof}

This then allows us to prove the analogous result to Theorem \ref{lbl7} for $SU(n)$:

\begin{theorem}
     Let $M$ be a closed, connected, smooth manifold of dimension d with fundamental group $\Gamma=\pi_{1}(M,m_0)$. Let $f:M\rightarrow B\Gamma(M,m_0)$ be a map classifying the universal cover $\widetilde{M}\rightarrow M$ that takes the basepoint $m_0$ to the canonical basepoint of $B\Gamma$. Let $\beta_{k}(M)=\mathrm{Rank}(H^{k}(M;\bbQ))$. Then for each $\rho\in \mathrm{Hom}(\Gamma,SU(n))$ and each $m\geq 0$, the induced map on homotopy groups $$(f^{*}\circ B)_{*}:\pi_{m}(\mathrm{Hom}(\Gamma,SU(n)),\rho)\rightarrow \pi_{m}(\mathrm{Map}_{*}(M,BSU(n)),B\rho \circ f)$$ satisfies $\mathrm{Rank}(\mathrm{coker}((f^{*}\circ B)_{*}))\geq \sum_{i=2}^{\infty}\beta_{2i+m}(M)$ when $n\geq (m+d)/2$.
\end{theorem}

\begin{proof}
    The proof is similar to that of Theorem \ref{lbl7}. Note that the map induced by inclusion  $BSU(n)\rightarrow BSU(n+1)$ is again $(2n+1)$-connected. The only significant difference is that Lemma \ref{lbl15} requires the first Chern class of the $U(n)$-bundle we are pulling back to be zero, so in this case the rank of the group $A=\mathrm{Span}_{\bbZ}(\{a_{ij}\}_{i,j})\subset \tilde{\mathrm{K}}^{0}(S^m\wedge M)$ is the sum of the Betti numbers $\beta_{2i+m}(M)$ now starting with $i=2$.
\end{proof}

We then get an analogous statement to Corollary \ref{lbl10}:

\begin{corollary}
     Let M be closed, smooth, connected manifold of dimension d, and let $E\rightarrow M$ be a flat principal $SU(n)$-bundle over M. Then for each $A_0\in \mathcal{A}_{\mathrm{flat}}(E)$ and each $m$ such that $0< m\leq 2n-d-1$, we have $$\mathrm{Rank}(\pi_{m}(\mathcal{A}_{\mathrm{flat}}(E),A_0))\geq \sum_{i=2}^{\infty}\beta_{m+2i+1}(M).$$ In addition, if $\sum_{i=2}^{\infty}\beta_{2i+1}>0$ and $2n-d-1\geq 0$, then $\mathcal{A}_{\mathrm{flat}}(E)$ has infinitely many path components. 
\end{corollary}

To analyze the case $G=SO(n)$ we again want to reduce the structure group of $O(n)$-bundles, and we then want to further reduce those to Spin($n$)-bundles. An $O(n)$-bundle $E$ reduces to a $SO(n)$-bundle if and only if the first Stiefel-Whitney class $w_{1}(E)\in H^{1}(X;\bbZ/2)$ is zero. An $SO(n)$-bundle $F$ reduces to a Spin($n$)-bundle if and only if $w_2(F)=0$ (see \cite{maycharclasses} Corollary 5.3 and Remark 6.11).

\begin{lemma}
     Let $X$ be a finite CW complex. Given $m>0$ and $x\in H^{4i}(S^m \wedge X;\bbQ)$, there exists a class $\phi\in \mathrm{KO}^{0}(S^m \wedge X)$ such that $p_i(\phi)=qx$ for some non-zero rational number $q\in \bbQ^{*}$, $p_j(\phi)=0$ for $j\neq i$, and $w_{1}(\phi)=0=w_{2}(\phi)$.  
\end{lemma}

\begin{proof}
    We need to show that the classes of bundles constructed in Lemma \ref{lbl11} can be chosen to have $w_{1}(\phi)=0=w_{2}(\phi)$. If the Chern character satisfies $Ch(\phi')=q'x$ for some chosen class $x\in H^{4i}(S^m \wedge X;\bbQ)$ and $\phi'\in \mathrm{KO}^{0}(S^m \wedge X)$, then $Ch(4\phi')=4q'x$ as $Ch$ is a ring homomorphism, and as $w_{1}(E\oplus F)=w_{1}(E)+w_{1}(F)$ then $w_{1}(4\phi')=0$. As \[w_2(E\oplus F)=w_2(E)+w_1(E)w_1(F)+w_2(F)\] we also see $w_{2}(4\phi')=0$. So let $\phi=4\phi'$ and $q=4q'$.
\end{proof}

Now we get the same results on the space of flat connections as we did with $O(n)$-bundles. Note that, unlike in the unitary case, we do not have to change the lower bound on the rank to reduce to $SO(n)$ or Spin($n$).

\begin{theorem}
    Theorem \ref{lbl16} and Corollary \ref{lbl17} hold for $SO(n)$ and Spin($n$)-bundles.
\end{theorem}

\section{Appendix}

In this appendix, we give a proof of Prop. \ref{lbl12}.
\begin{proof}[Proof of Prop. \ref{lbl12}]

Given a real vector bundle $E\rightarrow X$, one can consider the complexification of $E$, given by $E\otimes_{\bbR}\bbC$. This induces a map $c:\mathrm{KO}^{0}(X)\rightarrow \mathrm{K}^{0}(X)$. It is straightforward to check this is a ring homomorphism. Let $r: \mathrm{K}^{0}(X)\rightarrow \mathrm{KO}^{0}(X)$ denote the forgetful map sending a complex vector bundle to its underlying real vector bundle, a group homomorphism. Let $\overline{E}$ denote the conjugate bundle of a complex vector bundle, defined as the vector bundle whose fiber $\overline{E}_x$ is the conjugate vector space to $E_{x}$. This induces a well-defined $\bbZ/2$-action on $\mathrm{K}^{0}(X)$. Let $\mathrm{K}^{0}_{I}(X)$ denote the set of elements invariant under this action. Note that $\mathrm{Im}(c)\subseteq \mathrm{K}^{0}_{I}(X)$, and one can show (see \cite{karoubi} Prop. 2.9) that $r'=r\lvert_{\mathrm{K}^{0}_{I}(X)}$ and $c$ satisfy $r'(c(x))=2x$ and $c(r'(y))=2y$. So $c\otimes 1$ is a ring isomorphism $\mathrm{KO}^{0}(X)\otimes_{\bbZ}\bbZ[1/2]\rightarrow \mathrm{K}^{0}_{I}(X)\otimes_{\bbZ}\bbZ[1/2]$ with inverse $r'\otimes 1/2$. This extends to an isomorphism $\mathrm{KO}^{0}(X)\otimes_{\bbZ}\bbQ\rightarrow \mathrm{K}^{0}_{I}(X)\otimes_{\bbZ}\bbQ$. \\

I now claim that the Chern character isomorphism restricts to an isomorphism $$\mathrm{K}^{0}_{I}(X)\otimes_{\bbZ} \bbQ \rightarrow \prod_{i\geq 0}H^{4i}(X;\bbQ)$$ To see this, consider the conjugation action on $\mathrm{K}^{0}(X)\otimes_{\bbZ} \bbQ$ and the $\bbZ/2$-action on $\prod_{i\geq 0}H^{2i}(X;\bbQ)$ defined on components by $$\overline{x}\in H^{2i}(X;\bbQ)= \begin{cases} 
      -x &  \text{if i odd} \\
      x & \text{if i even} 
   \end{cases}
$$ \\

The Chern character is defined on vector bundle $E\rightarrow X$ by $$\mathrm{Ch}(E)=\dim(E)+\sum_{k>0}s_{k}(c_1(E),\dots,c_{k}(E))/k!$$ where $c_i(E)\in H^{2i}(X;\bbZ)$ is the $i$-th Chern class and $s_k$ is the $k$-th Newton polynomial $$s_k(c_1,\dots,c_k)=k(-1)^k\sum_{\substack{r_1+2r_2+\dots+mr_m=m \\r_1\geq 0,\dots,r_m\geq 0}}\dfrac{(r_1+r_2+\dots +r_m-1)!}{r_{1}!r_{2}!\dots r_{m}!}\prod_{i=1}^{m}(-c_i)^{r_i}$$

Using the fact that $c_{i}(\overline{E})=(-1)^{i}c_{i}(E)$ (see \cite{milnorstasheff}), it is easy to verify that the Chern character is equivariant with respect to the $\bbZ/2$-actions above. Thus Ch restricts to an isomorphism on the fixed sets of the $\bbZ/2$-action, giving our desired isomorphism. \end{proof}

\bibliographystyle{plain} 
\bibliography{refs} 

\end{document}